\documentclass[12pt,reqno]{amsart}
\usepackage{amsthm,amsfonts,amssymb,euscript,fancyhdr}

\usepackage{graphicx} 
\usepackage[font=small,labelfont=bf]{caption} 

\newcommand{\bl}{\begin{lemma}}
\newcommand{\el}{\end{lemma}}
\def\beaa{\begin{eqnarray*}}
\def\eeaa{\end{eqnarray*}}
\def\ba{\begin{array}}
\def\ea{\end{array}}
\def\be#1{\begin{equation} \label{#1}}
\def \eeq{\end{equation}}

\def\a{{\alpha}}

\def\b{{\beta}}
\def\be{{\beta}}

\def\de{\delta}

\def\la{\lambda}

\def\Si{\Sigma}

\def\Om{\Omega}

\def\Th{\Theta}

\def\nab{\nabla}
\def\varep{\varepsilon}

\def\pr{{\partial}}

\def\AA{{\mathcal A}}

\def\CC{{\mathcal C}}
\def\MM{{\mathcal M}}

\def\HH{{\mathcal H}}
\def\LL{{\mathcal L}}

\def\DD{{\mathcal D}}

\def\D{{\bf D}}

\def\g{{\bf g}}

\def\HHH{{\mathbb{H}}}

\def\RRR{{\mathbb R}}

\def\f12{{\frac 1 2}}

\DeclareMathOperator{\Div}{\mathrm{div}}

\def\half{\frac{1}{2}}

\newcommand{\RRRic}{\mathrm{Ric}}

\newcommand{\ol}{\overline}

\def\ni{\noindent}

\def\tr{\mathrm{tr}}

\def\f{\widetilde{f}}

\newcommand{\ubp}{\underline{B}^+}

\newtheorem{theorem}{Theorem}[section]
\newtheorem{lemma}[theorem]{Lemma}
\newtheorem{proposition}[theorem]{Proposition}

\newtheorem{definition}[theorem]{Definition}
\newtheorem{remark}[theorem]{Remark}
\newtheorem{conjecture}[theorem]{Conjecture}

\numberwithin{equation}{section}
\begin{document}

\title{The localised bounded $L^2$-curvature theorem}

\author{Stefan Czimek}

\address{The Fields Institute for Research in Mathematical Sciences, Toronto, Canada}

\begin{abstract} In this paper, we prove a localised version of the bounded $L^2$-curvature theorem of Klainerman-Rodnianski-Szeftel \cite{KRS}. More precisely, we consider initial data for the Einstein vacuum equations posed on a compact spacelike hypersurface $\Si$ with boundary, and show that the time of existence of a classical solution depends only on an $L^2$-bound on the Ricci curvature, an $L^4$-bound on the second fundamental form of $\pr \Si \subset \Si$, an $H^1$-bound on the second fundamental form, and a lower bound on the volume radius at scale $1$ of $\Si$. \newline
Our localisation is achieved by first proving a localised bounded $L^2$-curvature theorem for small data posed on $B(0,1)$, and then using the scaling of the Einstein equations and a low regularity covering argument on $\Si$ to reduce from large data on $\Si$ to small data on $B(0,1)$. The proof uses the author's previous work \cite{Czimek1}, \cite{Czimek21}, and the bounded $L^2$-curvature theorem \cite{KRS} as black boxes.
\end{abstract}

\maketitle
\tableofcontents

%%%%%%%%%%%%%%%%%%%%%%%%%%%%%%%%%%%%%%%%%%%%%%%%%%%%%%%%%%

\section{Introduction} \label{sec:introduction}

\subsection{The Cauchy problem of general relativity}

A Lorentzian $4$-manifold $(\mathcal{M},{\bf g})$ solves the \emph{Einstein vacuum equations} if
\begin{align} \label{EinsteinVacuumEquationsIntroJ}
\mathbf{Ric} = 0,
\end{align}
where $\mathbf{Ric}$ denotes the Ricci tensor of the Lorentzian metric ${ \bf g}$. The Einstein vacuum equations are invariant under diffeomorphisms, and therefore one considers equivalence classes of solutions. Expressed in general coordinates, \eqref{EinsteinVacuumEquationsIntroJ} is a non-linear geometric coupled system of partial differential equations of order $2$ for ${\bf g}$. In a suitable gauge, namely the so-called \emph{wave coordinates}\footnote{On a Lorentzian $4$-manifold $(\MM,\mathbf{g})$, wave coordinates $(x^0,x^1,x^2,x^3)$ satisfy by definition
\begin{align*}
\square_{\mathbf{g}} x^\a =0 \, \text{ for } \a=0,1,2,3.
\end{align*}
The Einstein equations reduce in wave coordinates to
\begin{align*}
\square_{\mathbf{g}} (\mathbf{g}_{\a \b}) = \mathcal{N}_{\a\b}(\mathbf{g},\pr_\mu \mathbf{g}), \, \text{ for } \a,\be=0,1,2,3,
\end{align*}
where $\mathcal{N}_{\a \b}(\mathbf{g}, \pr \mathbf{g})$ is a non-linearity that is linear in $\mathbf{g}$ and quadratic in $\pr_\mu \mathbf{g}$, $\mu=0,1,2,3$.}, it can be shown that \eqref{EinsteinVacuumEquationsIntroJ} is hyperbolic and hence corresponds to an evolution problem. \\

Initial data for the Einstein vacuum equations is specified by a triple $(\Si,g,k)$ where $(\Si,g)$ is a Riemannian $3$-manifold and $k$ a symmetric $2$-tensor on $\Si$ satisfying the \emph{constraint equations},
\begin{align} \begin{aligned} 
\mathrm{R}_{scal} &= \vert k \vert_g^2 - (\tr_g k )^2, \\
\Div_g k &= d(\tr_g k). \label{eq:CONST1}
\end{aligned} \end{align}
Here $\mathrm{R}_{scal}$ and $d$ denote the scalar curvature of $g$ and the exterior derivative on $\Si$, respectively, and
\begin{align*}
\vert k \vert_g^2:= g^{ij} g^{lm} k_{il} k_{jm}, \, \tr_g k: = g^{ij}k_{ij}, \, (\Div_g k )_l := g^{ij} \nab_i k_{jl},
\end{align*}
where $\nab$ denotes the covariant derivative on $(\Si,g)$ and we tacitly use, as in the rest of this paper, the Einstein summation convention.\\

In the seminal \cite{ChoquetBruhatFirst} it is shown that the above initial value formulation is well-posed. In the future development $(\mathcal{M},{\bf g})$ of given initial data $(\Si,g,k)$, the $3$-manifold $\Si \subset \mathcal{M}$ is a spacelike Cauchy hypersurface with induced metric $g$ and second fundamental form $k$. See for example \cite{Wald} or \cite{Ringstrom} for details. \\

In the rest of this paper, we assume that the initial hypersurface $\Si$ is \emph{maximal}, that is, 
$$\tr_g k =0 \,\,  \text{ on } \Si.$$ 
This assumption is sufficiently general for our purposes, see \cite{Bartnik1}. In particular, on a maximal hypersurface $\Si$, the constraint equations \eqref{eq:CONST1} reduce to the \emph{maximal constraint equations},
 \begin{align*}
\mathrm{R}_{scal} &= \vert k \vert_g^2, \\
\Div_g k &= 0,\\
\tr_g k&=0.
\end{align*}

%%%%%%%%%%%%%%%%%%%%%%%%%%%%%%%%%%%%%%%%%%%%%%%%%%%%%%%%%%

\subsection{Weak cosmic censorship and the bounded $L^2$-curvature theorem}

One of the main open questions of mathematical relativity is the so-called \emph{weak cosmic censorship conjecture} formulated by Penrose, see \cite{PenroseConjecture}.
\begin{conjecture}[Weak cosmic censorship conjecture] \label{conj2}
For a generic solution to the Einstein equations, all singularities forming in the context of gravitational collapse are covered by black holes.
\end{conjecture}

In the ground-breaking \cite{Chr9}, Christodoulou proves Conjecture \ref{conj2} for the Einstein-scalar field equations under the assumption of spherical symmetry. In Christodoulou's proof, a low regularity control of the Einstein equations is essential for analysing the dynamical formation of black holes. More precisely, in \cite{Chr7} Christodoulou proves a well-posedness result for initial data which is bounded only in a scale-invariant BV-norm, and subsequently uses this framework to establish the formation of trapped surfaces in \cite{Chr9}. \\

The result in \cite{Chr9} strongly suggests that a crucial step to prove the weak cosmic censorship in the absence of symmetry is to control solutions to the Einstein vacuum equations in very low regularity\footnote{Note that bounded variation norms are not suitable outside of spherical symmetry. In the absence of spherical symmetry, regularity should be measured with respect to $L^2$-based spaces, see \cite{Stein}.}. The current state-of-the-art with respect to low regularity control of solutions to the Einstein vacuum equations is the \emph{bounded $L^2$-curvature theorem} by Klainerman-Rodnianski-Szeftel, see Theorem 2.2 in \cite{KRS}. We refer to the introduction of \cite{KRS} for a historical account of the developments leading to this result. 

%%%%%%%%%%%%%%%%%%%%%%%%%%%%%%%%%%%%%%%%%%%%%%%%%%%%%%%%%%
\begin{theorem}[The bounded $L^2$-curvature theorem, \cite{KRS}] \label{thmbl2intro1largedata}
Let $(\mathcal{M}, {\bf g})$ be an asymptotically flat solution to the Einstein vacuum equations together with a maximal foliation by space-like hypersurfaces $\Si_t$ defined as level sets of a time function $t$. Assume that the initial slice $(\Si_0, g,k)$ is such that $\Si_0 \simeq \RRR^3$ and 
\begin{align}\begin{aligned}
\Vert \RRRic \Vert_{L^2(\Si_0)} < \infty, \, \Vert \nab k \Vert_{L^2(\Si_0)} < \infty \text{ and } r_{vol}(\Si_0,1) >0,
\end{aligned} \label{eq:conditionsepL2} \end{align}
where $r_{vol}(\Si_0,1)$ is the volume radius\footnote{The volume radius of $(\Si_0,g)$ at scale $1$ is defined as
\begin{align*}
r_{vol}(\Si_0,1) := \inf_{p \in \Si_0} \inf_{0<r<1} \frac{\mathrm{vol}_g \left( B_g(p,r) \right)}{\frac{4 \pi}{3}r^3},
\end{align*}
where $B_g(p,r)$ denotes the geodesic ball of radius $r$ centered at the point $p$.} of $(\Si_0,g)$ at scale $1$, and $\RRRic$ denotes the Ricci tensor of $g$. Then 
\begin{enumerate}
\item {\bf $L^2$-regularity.} There exists a time 
\begin{align*}
T=T(\Vert \RRRic \Vert_{L^2(\Si_0)}, \Vert \nab k \Vert_{L^2(\Si_0)}, r_{vol}(\Si_0,1)) >0, 
\end{align*}
and a constant
\begin{align*}
C=C(\Vert \RRRic \Vert_{L^2(\Si_0)}, \Vert \nab k \Vert_{L^2(\Si_0)}, r_{vol}(\Si_0,1)) >0, 
\end{align*}
such that the following control holds on $0 \leq t \leq T$.
\begin{align*}
\Vert {\bf R} \Vert_{L^\infty_t L^2(\Si_t)} \leq C,\, \Vert \nab k_t \Vert_{L^\infty_t L^2(\Si_t)} \leq C, \, \inf\limits_{0\leq t \leq T}  r_{vol}(\Si_t, 1) \geq \frac{1}{C},
\end{align*}
where ${\bf R}$ denotes the Riemann curvature tensor of $(\MM,\g)$, and $g_t$ and $k_t$ are the induced metric and the second fundamental form of $\Si_t$, respectively.
\item {\bf Higher regularity.} In case of higher regularity of the initial data, we have for integers $m\geq1$, within the same time interval as in part (1), the higher derivative estimate
\begin{align*}
\sum\limits_{\vert \a \vert \leq m } \Vert \D^{(\a)} \mathbf{R} \Vert_{L^{\infty}_{t}L^2(\Si_t)} \leq C_m \sum\limits_{\vert i \vert \leq m} \Big( \Vert \nab^{(i)} \RRRic \Vert_{L^2(\Si_0)}+ \Vert \nab^{(i)} \nab k \Vert_{L^2(\Si_0)} \Big),
\end{align*}
where the constant $C_m>0$ depends only on the previous $C$ and $m$.
\end{enumerate}
\end{theorem}
%%%%%%%%%%%%%%%%%%%%%%%%%%%%%%%%%%%%%%%%%%%%%%%%%%%%%%%%%%

\begin{remark} In the above theorem, as in the rest of this paper, the statement should be understood as a continuation result for smooth solutions. That is, a solution to the Einstein vacuum equation developed from smooth initial data can smoothly be continued as long as condition \eqref{eq:conditionsepL2} holds. For details, see the introduction in \cite{KRS}. \end{remark}

The proof of Theorem \ref{thmbl2intro1largedata} is based on bilinear estimates, see \cite{KRS}, as well as Strichartz estimates, see \cite{J5}, in a low regularity spacetime where the Riemann curvature tensor is only assumed to be in $L^2$. The proof of these estimates relies crucially on a plane wave representation formula for the wave equation on low regularity spacetimes constructed in \cite{J1}-\cite{J5}. This plane wave representation formula is built as a Fourier integral operator which necessitates the assumption $\Si_0 \simeq \RRR^3$. \\

Given that the Einstein equations are hyperbolic and have finite speed of propagation, the assumption $\Si_0 \simeq \RRR^3$ in Theorem \ref{thmbl2intro1largedata} seems unnatural. Furthermore, gravitational collapse is studied in local domains of dependence, that is, given a \emph{compact} initial data set with boundary, one considers the development inside the future domain of dependence of the initial data set, see for example \cite{Chr9} and \cite{ChrFormationNonSpherical}. For these reasons, it appears important to localise Theorem \ref{thmbl2intro1largedata}, that is, to relax the condition $\Si_0 \simeq \RRR^3$, which is the main goal of this paper.

%%%%%%%%%%%%%%%%%%%%%%%%%%%%%%%%%%%%%%%%%%%%%%%%%%%%%%%%%%

\subsection{The localised bounded $L^2$ curvature theorem}
The following is the main result of this paper.
\begin{theorem}[The localised bounded $L^2$-curvature theorem] \label{thm:bl2locintro1} Let $(\Si,g,k)$ be a maximal initial data set such that $(\Si,g)$ is a compact complete\footnote{A smooth Riemannian manifold with boundary is called complete if it is complete as a metric space. } smooth Riemannian manifold with boundary and assume that
\begin{align*}
\Vert \RRRic \Vert_{L^2(\Si)} < \infty, \,\, \Vert k \Vert_{L^4(\Si)} < \infty, \,\, \Vert \nab k \Vert_{L^2(\Si)} < \infty, \,\,
\Vert \Th \Vert_{L^4(\pr \Si)} < \infty, \,\, r_{vol}(\Si,1)>0. \end{align*}
where $\Th$ denotes the second fundamental form of $\pr \Si \subset \Si$. Then,
\begin{enumerate}
\item {\bf $L^2$-regularity.} There exists a radius
\begin{align*}
r=r(\Vert \RRRic \Vert_{L^2(\Si)}, \Vert k \Vert_{L^4(\Si)},\Vert \nab k \Vert_{ L^2(\Si)},\Vert \Th  \Vert_{ L^4(\pr \Si)},r_{vol}(\Si,1)  )>0,
\end{align*}
a time
\begin{align*}
T=T(\Vert \RRRic \Vert_{L^2(\Si)}, \Vert k \Vert_{L^4(\Si)},\Vert \nab k \Vert_{ L^2(\Si)},\Vert \Th  \Vert_{ L^4(\pr \Si)},r_{vol}(\Si,1)  )>0,
\end{align*}
and a constant
\begin{align*}
C=C(\Vert \RRRic \Vert_{L^2(\Si)}, \Vert k \Vert_{L^4(\Si)},\Vert \nab k \Vert_{ L^2(\Si)},\Vert \Th  \Vert_{ L^4(\pr \Si)},r_{vol}(\Si,1)  )>0,
\end{align*}
such that for every point $p \in \Si$, the future domain of dependence $\DD$ of the geodesic ball $B_g(p,r)$ admits a time function $t$ whose level sets $\Si_t$ are spacelike maximal hypersurfaces and foliate $\DD$ with $\Sigma_0=B_g(p,r) \subset \Si$, and the following control holds on $0\leq t\leq T$,
\begin{align*}
\Vert {\bf \mathbf{R}} \Vert_{L^\infty_{t} L^2(\Si_{t})} \leq C, \Vert k_t \Vert_{L^\infty_t L^4(\Si_t)} \leq C,  \Vert \nab k_t \Vert_{L^\infty_{t} L^2(\Si_{t})} \leq C, \inf\limits_{0\leq t \leq T} r_{vol}(\Si_{t},r) \geq \frac{1}{C}.
\end{align*}
\item {\bf Higher regularity.} In case of higher regularity, we have for $m\geq1$, within the same time interval as in part (1), the higher derivative estimate
\begin{align*}
\sum\limits_{\vert \a \vert \leq m } \Vert \D^{(\a)} \mathbf{R} \Vert_{L^{\infty}_{t}L^2(\Si_t)} \leq C_m \sum\limits_{\vert i \vert \leq m} \Big( \Vert \nab^{(i)} \RRRic \Vert_{L^2(\Si)}+ \Vert \nab^{(i)} \nab k \Vert_{L^2(\Si)} +1 \Big),
\end{align*}
where the constant $C_m>0$ depends only on the previous $C$ and $m$.
\end{enumerate}
\end{theorem}

%%%%%%%%%%%%%%%%%%%%%%%%%%%%%%%%%%%%%%%%%%%%%%%%%%%%%%%%%%

\subsection{Overview of the proof of Theorem \ref{thm:bl2locintro1}} \label{subsec35353} 

In this section, we sketch the proof of Theorem \ref{thm:bl2locintro1} in three steps.

\begin{enumerate}

\item {\bf A localised bounded $L^2$-curvature theorem for small data on $B(0,1)$.} In Section \ref{L2onballssmall}, we prove that for sufficiently small initial data $(g,k)$ in $H^2 \times H^1$ on $B(0,1) \subset \RRR^3$, the future domain of dependence of $(B(0,1),g,k)$ is well-controlled up to time $T=1/2$, see the precise statement in Proposition \ref{prop:red}.

%%%%%%%%%%%%%%%%%%%%%%%%%%%%%%%%%%%%%%%%%%%%%%%%%%%%%%%%%%

\ni The proof of Proposition \ref{prop:red} follows by an application of Theorem \ref{thmbl2intro1largedata} and the \emph{extension procedure for the constraint equations} \cite{Czimek1} as black boxes.

\item {\bf Construction of a cover of $\Si$ by coordinate systems $\ol{B(0,r_1)}$ and $B(0,r_2)$.} In Section \ref{subsec5893434}, we cover $\Si$ by coordinate systems $\ol{B(0,r_1)}$ and $B(0,r_2)$, where the radii $r_1,r_2>0$ are chosen small and depend only on low regularity bounds assumed in Theorem \ref{thm:bl2locintro1}.

\ni The cover is constructed such that its \emph{Lebesgue number}\footnote{Given a cover $(C_i)_{i \in I}$ of $\Si$, its Lebesgue number $\ell$ is defined as the largest number such that for each point $p \in \Si$, the geodesic ball $B_g(p,\ell)$ is completely contained in $C_i$ for some $i\in I$.} is bounded from below and depends only on low regularity bounds assumed in Theorem \ref{thm:bl2locintro1}. 

 \ni The construction uses the \emph{existence of boundary harmonic coordinates on manifolds with boundary} \cite{Czimek21} as black box. We remark that the coordinate systems $\ol{B(0,r_1)}$ cover an open neighbourhood of $\pr \Si$ in $\Si$, and the $B(0,r_2)$ cover the rest of $\Si$.

\item {\bf Scaling to small data.} In Section \ref{sec:ReductionToSmallBalls}, we use the scaling invariance of the Einstein equations to rescale for $\la>0$
\begin{align*} 
(B(0,r),g,k) \to (B(0,\la^{-1}r),g_{(\la)},k_{(\la)}).
\end{align*}
We show that for $\la>0$ sufficiently small, depending only on low regularity bounds assumed in Theorem \ref{thm:bl2locintro1}, the rescaled initial data is small in $H^2 \times H^1$, see Lemma \ref{lemmaScalingSmall}.
\end{enumerate}

The proof of Theorem \ref{thm:bl2locintro1} is then concluded in Section \ref{SectionConclusionOfTheorem} by combining the above three steps.

%%%%%%%%%%%%%%%%%%%%%%%%%%%%%%%%%%%%%%%%%%%%%%%%%%%%%%%%%%

\subsection{Overview of the paper}
In Section \ref{sec:definitions}, we introduce notations and preliminaries. In Sections \ref{L2onballssmall}-\ref{sec:ReductionToSmallBalls}, we prove Steps (1)-(3) as outlined above. In Section \ref{SectionConclusionOfTheorem}, we conclude the proof of Theorem \ref{thm:bl2locintro1}.

%%%%%%%%%%%%%%%%%%%%%%%%%%%%%%%%%%%%%%%%%%%%%%%%%%%%%%%%%%

\subsection{Acknowledgements} This work forms part of my Ph.D. thesis. I am grateful to my Ph.D. advisor J\'er\'emie Szeftel for his kind supervision and careful guidance. This work is financially supported by the RDM-IdF.

%%%%%%%%%%%%%%%%%%%%%%%%%%%%%%%%%%%%%%%%%%%%%%%%%%%%%%%%%%

\section{Notations, definitions and prerequisites} \label{sec:definitions}
In this section, we introduce notations, definitions and preliminary results that are used in this paper.  \\

%%%%%%%%%%%%%%%%%%%%%%%%%%%%%%%%%%%%%%%%%%%%%%%%%%%%%%%%%%

In this work, lowercase Latin indices run through $i,j=1,2,3$. Greek indices run through $\mu, \nu = 0,1,2, 3$. We write $ A \lesssim B$ if there exists a universal constant $C>0$ such that $A \leq C B$. \\

Let the closed upper half-space of $\RRR^3$ be denoted by
$$\HHH^+ := \Big\{ x \in \RRR^3 \Big\vert x^3 \geq 0 \Big\}.$$ 
For a point $x\in \HHH^+$ and a real number $r>0$, let
\begin{align*} 
B(x,r) := \Big\{ y \in \RRR^3 \Big\vert \vert x-y \vert < r \Big\}, \,\, B^+(x,r) := B(x,r) \cap \HHH^+.
\end{align*}  

%%%%%%%%%%%%%%%%%%%%%%%%%%%%%%%%%%%%%%%%%%%%%%%%%%%%%%%%%%

\begin{definition}[Function spaces]
Let $m\geq1$ be an integer. Let $\Om \subset \RRR^3$ be an open subset, and let $f$ be a scalar function on $\Om$.
\begin{enumerate}
\item Let the norm
\begin{align*}
\Vert f \Vert^p_{H^m(\Om)} := \sum\limits_{\vert \a \vert\leq m} \Vert \pr^\a f \Vert^2_{L^2(\Om)},
\end{align*}
and define the function space $H^m(\Om)$ by
\begin{align*}
H^m(\Om) := \left\{ f \in L^2(\Om): \Vert f \Vert^2_{H^m(\Om)} < \infty \right\}.
\end{align*}
\item Let the norm
\begin{align*}
\Vert f \Vert_{C^m(\ol{\Om})} := \max\limits_{\vert \be \vert \leq m}  \, \sup\limits_{x\in \ol{\Om}} \left\vert \pr^\be f \right\vert.
\end{align*}
and define $C^m(\ol{\Om})$ to be the function space of $m$-times differentiable functions on $\Om$ equipped with norm $\Vert \cdot \Vert_{C^m(\ol{\Om})}$.
\end{enumerate}
Here $\a=(\a_1,\a_2, \a_3)$, $\be=(\be_1, \be_2, \be_3) \in \mathbb{N}^3$, and $\pr^{\a}:= \pr_1^{\a_1}\pr_{2}^{\a_2} \pr_3^{\a_3}$, $\vert \a \vert:= \sum_{i=1}^3 \vert \a_i \vert$. 
\end{definition}

\begin{definition}[Tensor spaces] Let $\Om \subset \RRR^3$ be an open subset and let $T$ be a tensor on $\Om$. For integers $m\geq1$ and reals $1<p<\infty$, define
\begin{align*}
\HH^{m}(\Om), \LL^p(\Om) \text{ and } \CC^{m}(\Om)
\end{align*}
to be the spaces of tensors whose coordinate components are respectively in 
\begin{align*}
H^{m}(\Om), L^p(\Om) \text{ and } C^{m}(\Om)
\end{align*}
equipped with the natural norm, that is, for example, for an $(l,m)$-tensor $T$ on $\Om$,
\begin{align*}
\Vert T \Vert_{\HH^m(\Om)} := \sum\limits_{i_1, \dots i_l=1}^3 \sum\limits_{j_1 \dots j_m=1}^3 \Vert T^{i_1 \dots i_{l}}_{\,\,\,\,\,\,\,\,\,\,\,\,\,\, j_1 \dots j_m} \Vert_{H^m(\Om)},
\end{align*}
where $T^{i_1 \dots i_{l}}_{\,\,\,\,\,\,\,\,\,\,\,\,\,\, j_1 \dots j_m}$ denote the coordinate components.
\end{definition}

%%%%%%%%%%%%%%%%%%%%%%%%%%%%%%%%%%%%%%%%%%%%%%%%%%%%%%%%%%

\begin{definition}[Volume radius at scale $r$] \label{def:volumeradius}
Let $(M,g)$ be a Riemannian $3$-manifold with boundary. For a real $r>0$ and a point $p \in M$, the \emph{volume radius at scale $r$ at $p$} is defined as
\begin{align*}
r_{vol}(r,p) := \inf\limits_{r'< r} \frac{\mathrm{vol}_g \left( B_g(p,r') \right)}{\frac{4\pi}{3} (r')^{3}}.
\end{align*}
The \emph{volume radius of $(M,g)$ at scale $r$} is defined as
\begin{align*}
r_{vol}(M,r) := \inf\limits_{p \in M} r_{vol}(r,p).
\end{align*}
\end{definition}

%%%%%%%%%%%%%%%%%%%%%%%%%%%%%%%%%%%%%%%%%%%%%%%%%%%%%%%%%%

\begin{definition}[Lebesgue number] \label{defLebesguenumber2435} Let $(M,g)$ be a Riemannian manifold. Given a covering $(C_i)_{i=1}^N$ of $M$, the Lebesgue number $\ell$ is defined to be the largest real number such that for each point $p \in M$, there is an $i\in \{1, \dots, N\}$ such that $B_g(p,\ell) \subset C_i$.
\end{definition}

%%%%%%%%%%%%%%%%%%%%%%%%%%%%%%%%%%%%%%%%%%%%%%%%%%%%%%%%%%
\section{The localised bounded $L^2$-curvature theorem for small data on $B(0,1)$} \label{L2onballssmall}

In this section, we prove the following result.
\begin{proposition}[Localised bounded $L^2$ curvature theorem for small data on $B(0,1)$] \label{prop:red}
Let $(\bar{g}, \bar{k})$ be maximal initial data for the Einstein vacuum equations on $B(0,1) \subset \RRR^3$ and assume that for some $\varep>0$, 
\begin{align*} 
\Vert \bar{g}-e \Vert_{\HH^2(B(0,1))} + \Vert \bar{k} \Vert_{\HH^{1}(B(0,1))}  < \varep.
\end{align*}
Let $(\DD,{\bf g})$ be the solution to the Einstein vacuum equations in the future domain of dependence $\DD$ of $B(0,1)$, and let $t$ be a time function in $\DD$ such that its level sets $\Si_{t}$ are spacelike maximal hypersurfaces and foliate $\DD$ with $\Si_0 = B(0,1)$. Then, the following holds.
\begin{enumerate}
\item {\bf $L^2$-regularity.} There is an $\varep_0>0$ small such that if $\varep< \varep_0$, then the following control holds on $0\leq t \leq 1/2$,
\begin{align*} 
\Vert \RRRic \Vert_{L^\infty_{t}L^2(\Si_t)} \lesssim \varep,  \Vert k_t \Vert_{L^\infty_{t}L^4(\Si_t)}\lesssim \varep, \Vert \nab k_t \Vert_{L^\infty_{t}L^2(\Si_t)} \lesssim \varep, \inf\limits_{0\leq t \leq 1/2}  r_{vol}(\Si_t,1) \geq 1/8.
\end{align*}
\item {\bf Higher regularity.} Let $m\geq1$ be an integer. In case of higher regularity, we have the following higher regularity estimate on $0 \leq t \leq 1/2$,
\begin{align*}
\sum\limits_{\vert \a \vert \leq m} \Vert \D^{(\a)} \mathbf{R} \Vert_{L^\infty_{t}L^2(\Si_t)} \leq C_m \Big(  \Vert \bar{g}-e \Vert_{\HH^{m+2}(B(0,1))} + \Vert \bar{k} \Vert_{\HH^{m+1}(B(0,1))} \Big),
\end{align*}
where the constant $C_m>0$ depends on $m$.
\end{enumerate}
\end{proposition}

The proof of Proposition \ref{prop:red} is based on the following two literature results.

\begin{theorem}[An extension procedure for the constraint equations, \cite{Czimek1}] \label{thm:czextension}
Let $(\bar{g},\bar{k})$ be maximal initial data for the Einstein vacuum equations on $B(0,1) \subset \RRR^3$ and assume that for some $\varep>0$ it holds that
\begin{align*}
\Vert \bar{g}_{ij}-e_{ij} \Vert_{\HH^{2}(B(0,1))} +\Vert \bar{k}_{ij} \Vert_{\HH^{1}(B(0,1))} < \varep.
\end{align*}
Then, the following holds.
\begin{enumerate}
\item {\bf $L^2$-regularity.} There is a universal $\varep_0>0$ such that if $\varep< \varep_0$, then there exists asymptotically flat maximal initial data $(g,k)$ on $\RRR^3$ with $(g,k)\vert_{B(0,1)} = (\bar{g}, \bar{k})$ and
\begin{align*} 
\Vert g-e \Vert_{\HH^{2}_{-1/2}(\RRR^3)} + \Vert k \Vert_{\HH^{1}_{-3/2}(\RRR^3)} \lesssim \Vert \bar{g}-e \Vert_{\HH^{2}(B(0,1))} + \Vert \bar{k} \Vert_{\HH^{1}(B(0,1))}.
\end{align*}
\item {\bf Higher regularity.} In case of higher regularity, we have for integers $m\geq1$ the following higher regularity estimates,
\begin{align*}
\Vert g-e \Vert_{\HH^{m+2}_{-1/2}(\RRR^3)}+ \Vert k \Vert_{\HH^{m+1}_{-3/2}(\RRR^3)} \leq C_m \Big(  \Vert \bar{g}-e \Vert_{\HH^{m+2}(B(0,1))} + \Vert \bar{k} \Vert_{\HH^{m+1}(B(0,1))} \Big),
\end{align*}
where the constant $C_m>0$ depends on $m$.
\end{enumerate}
\end{theorem}
Here $\HH^{m+2}_{-1/2}(\RRR^3)$ and $\HH^{m+1}_{-3/2}(\RRR^3)$ denote Sobolev spaces of tensors equipped with weights $-1/2$ and $-3/2$, respectively, corresponding to the asymptotic flatness of the initial data, see \cite{Czimek1} for details.

\begin{theorem}[The bounded $L^2$-curvature theorem for small data, \cite{KRS}] \label{thmL2smalldata}
Let $(\mathcal{M}, {\bf g})$ an asymptotically flat solution to the Einstein vacuum equations together with a maximal foliation by space-like hypersurfaces $\Si_t$ defined as level hypersurfaces of a time function $t$. Assume that the initial slice $(\Si_0, g,k)$ is such that $\Si_0 \simeq \RRR^3$ and 
\begin{align*}
\Vert \RRRic \Vert_{L^2(\Si_0)} \leq \varep, \, \Vert \nab k \Vert_{L^2(\Si_0)} \leq \varep \text{ and } r_{vol}(\Si_0,1) \geq \half.
\end{align*}
Then, 
\begin{enumerate}
\item {\bf $L^2$-regularity.} There exists a small universal constant $\varep_0>0$ such that if \newline $0< \varep < \varep_0$, then the following control holds on $0 \leq t \leq 1$,
\begin{align*}
\Vert {\bf R} \Vert_{L^\infty_{t}L^2(\Si_t)} \lesssim \varep,\, \Vert k \Vert_{L^\infty_t L^2(\Si_t)} \lesssim \varep, \, \Vert \nab k \Vert_{L^\infty_{t}L^2(\Si_t)} \lesssim \varep, \, \inf\limits_{0\leq t \leq 1}  r_{vol}(\Si_t,1) \geq \frac{1}{4}.
\end{align*}
\item {\bf Higher regularity.} In case of higher regularity, we have for integers $m\geq1$ the following higher regularity estimates on $0\leq t \leq 1$,
\begin{align*}
\sum\limits_{\vert \a \vert \leq m} \Vert \D^{(\a)} \mathbf{R} \Vert_{L^\infty_{t}L^2(\Si_t)} \lesssim \sum\limits_{\vert i \vert \leq m} \Big( \Vert \nab^{(i)} \RRRic \Vert_{L^2(\Si_0)} + \Vert \nab^{(i)} \nab k \Vert_{L^2(\Si_0)} \Big).
\end{align*}
\end{enumerate}
\end{theorem}

%%%%%%%%%%%%%%%%%%%%%%%%%%%%%%%%%%%%%%%%%%%%%%%%%%%%%%%%%%

For the rest of this section, we prove Proposition \ref{prop:red}. For $\varep>0$ sufficiently small, the initial data $(\bar{g},\bar{k})$ on $B(0,1)$ can be extended by Theorem \ref{thm:czextension} to an asymptotically flat, maximal initial data set $(g,k)$ on $\RRR^3$ such that $(g,k)\vert_{B_1} = (\bar{g}, \bar{k})$ and
\begin{align*}
\Vert g-e \Vert_{\HH^2_{-1/2}(\RRR^3)} + \Vert k \Vert_{\HH^1_{-3/2}(\RRR^3)} \lesssim \varep.
\end{align*}
In particular, for $\varep>0$ sufficiently small, the extension $(g,k)$ on $\RRR^3$ satisfies the assertions of Theorem \ref{thmL2smalldata}.\\

Let therefore $(\mathcal{M}, {\bf g})$ denote the future development of $(\RRR^3,g,k)$ and $t$ be the time function in $\mathcal{M}$ such that its level sets $\Sigma_t$ are spacelike maximal hypersurfaces with $\Si_0 = \RRR^3$. By Theorem \ref{thmL2smalldata}, we have on $0\leq t \leq 1$,
 \begin{align*} 
 \Vert \RRRic_t \Vert_{L^{\infty}_t L^2(\Si_t)} \lesssim \varep, \Vert k_t \Vert_{L^\infty_t L^2(\Si_t)} \lesssim \varep, \Vert \nab k_t \Vert_{L^{\infty}_t L^2( \Si_t)} \lesssim \varep,\inf\limits_{0\leq t \leq 1}  r_{vol}(\Si_t,1) \geq 1/4,
\end{align*}
where $\RRRic_t$ denotes the Ricci curvature of the induced metric $g_t$ and $k_t$ the second fundamental form of $\Si_t$.\\

By restricting to the domain of dependence $\DD$ of $B(0,1)$, it follows that for $0\leq t \leq 1/2$,
\begin{align*}
&\Vert \RRRic_t \Vert_{L^{\infty}_tL^2(\Si_t \cap \DD)} \lesssim \varep, \Vert k_t \Vert_{L^{\infty}_t L^4( \Si_t \cap \DD)} \lesssim \varep, \Vert \nab k_t \Vert_{L^{\infty}_t L^2( \Si_t \cap \DD)} \lesssim \varep, \\
&\text{ and } \inf\limits_{0\leq t \leq 1/2}  r_{vol}(\Si_t \cap \DD,1) \geq 1/8.
\end{align*}
We remark that the control of the volume radius follows as in the proof of Theorem \ref{thmL2smalldata} by a control of $\bf{g}_{\mu \nu}$ in $C^0$; for details we refer the reader to the estimates in Section 4 of \cite{J3}.\\

It remains to prove the higher regularity estimate of Proposition \ref{prop:red}. By the higher regularity estimates of Theorems \ref{thm:czextension} and \ref{thmL2smalldata}, we have, for integers $m\geq1$, on $0\leq t \leq 1/2$,
\begin{align*}
\sum\limits_{\vert \a \vert \leq m} \Vert \D^{(\a)} \mathbf{R} \Vert_{L^\infty_t L^2(\Si_t)}\lesssim& \sum\limits_{\vert i \vert \leq m}\Big( \Vert \nab^{(i)} \RRRic \Vert_{L^2(\RRR^3)} + \Vert \nab^{(i)} \nab k \Vert_{L^2(\RRR^3)} \Big)\\
\lesssim& \Vert g-e \Vert_{\HH^{m+2}_{-1/2}(\RRR^3)} +\Vert k \Vert_{\HH^{m+1}_{-3/2}(\RRR^3)} \\
\leq& C_m  \Big( \Vert \bar{g}-e \Vert_{\HH^{m+2}(B(0,1))} +\Vert \bar{k} \Vert_{\HH^{m+1}(B(0,1))} \Big),
\end{align*}
where the constant $C_m>0$ depends on $m$. Restriction to the future domain of dependence $\DD$ of $B(0,1)$ then proves the higher regularity estimates of Proposition \ref{prop:red}. This finishes the proof of Proposition \ref{prop:red}.

%%%%%%%%%%%%%%%%%%%%%%%%%%%%%%%%%%%%%%%%%%%%%%%%%%%%%%%%%%

\section{Construction of the cover of $\Si$ by coordinate systems} \label{subsec5893434} 

In this section, we cover of $\Si$ by coordinate systems $\ol{B(0,r_1)}$ and $B(0,r_2)$, where the radii $r_1,r_2>0$ are small, depending only on the low regularity geometric bounds assumed in Theorem \ref{thm:bl2locintro1}. In particular, the radii $r_1$ and $r_2$ are sufficiently small such that subsequent rescaling to the unit ball leads to small data, see Sections \ref{sec:ReductionToSmallBalls} and \ref{SectionConclusionOfTheorem}.\\

At first, we construct coordinate systems $\ol{B(0,r_1)}$ near the boundary $\pr \Si$ of $\Si$. Then, we cover the rest of $\Si$ by coordinate systems $B(0,r_2)$. Finally, we prove that the Lebesgue constant $\ell$ of the constructed cover of $\Si$ is bounded from below.\\

The construction is based on the following existence result from \cite{Czimek21}.
\begin{theorem}[Existence of regular coordinate systems] \label{CGexistencetheorem}
Let $(M,g)$ be a smooth Riemannian $3$-manifold with boundary such that
\begin{align*}
\Vert \RRRic \Vert_{L^2(M)} < \infty, \Vert \Theta \Vert_{L^4(\pr M)} < \infty, r_{vol}(M,1)>0.
\end{align*}
Then, the following holds.
\begin{enumerate}
\item {\bf $L^2$-regularity.} There is $\varep_0>0$ such that for all $0< \varep< \varep_0$, there exists a radius
\begin{align*}
r=r(\Vert \RRRic \Vert_{L^2(M)}, \Vert \Theta \Vert_{L^4(\pr M)}, r_{vol}(M,1), \varep) >0
\end{align*}
such that for every $p \in \Si$, there is a chart $\varphi: B^+(x,r) \to U \subset M$ with $\varphi(x)=p$ such that
\begin{align} \label{Linftyboundsg2222}
(1-\varep) e_{ij} \leq g_{ij} \leq (1+ \varep) e_{ij}
\end{align}
and
\begin{align*} 
r^{-1/2} \Vert \pr g \Vert_{\LL^{2}(B^+(x,r))} + r^{1/2} \Vert \pr^2 g \Vert_{\LL^2(B^+(x,r))} \leq \varep.
\end{align*}
\item {\bf Higher regularity.} In case of higher regularity, we have for integers $m\geq1$,
\begin{align*} 
\Vert g \Vert_{\HH^{m+2}(B^+(x,r))} \leq C_r \sum\limits_{i=0}^m \Vert \nab^{(m)} \RRRic \Vert_{L^2(M)} + C_{r,m} \varep.
\end{align*}
\end{enumerate}

\end{theorem}

\begin{remark} \label{RemarkComparison} The bound \eqref{Linftyboundsg2222} allows to compare geodesic length on $M$ with coordinate length in the chart $\varphi$. In particular, it holds that for $r'<r$ and $\varep>0$ small,
\begin{align*}
B_g(p,(1-\varep) r') &\subset \varphi(B^+(x,r')) \subset B_g(p,(1+\varep) r'), \\
\varphi(B^+(x,(1-\varep) r')) &\subset B_g(p,r') \subset \varphi(B^+(x,(1+\varep) r')).
\end{align*}
\end{remark}

{\bf Construction of coordinate systems $\ol{B(0,r_1)}$ near $\pr \Si$.} First, for a given point $p \in \pr \Si$, we construct a coordinate system $\ol{B(0,r_1)}$ in $\Si$ containing $p$. Then we pick points $(p_i)_{i=1}^{N_1} \subset \pr \Si$ such that the corresponding constructed coordinate systems $\ol{B(0,r_1)}$ cover an open neighbourhood of $\pr \Si$ in $\Si$. \\

Let thus $p\in \pr \Si$, and let $\varep>0$ small to be determined. By Theorem \ref{CGexistencetheorem}, there is a radius
$$r=r(\Vert \RRRic \Vert_{L^2(M)}, \Vert \Theta \Vert_{L^4(\pr M)}, r_{vol}(M,1), \varep)>0$$
 and a chart $\varphi: B^+(0,r) \to U \subset \Si$ with $\varphi(0)=p$ and 
\begin{align} \begin{aligned} 
&(1-\varep) e_{ij} \leq g_{ij} \leq (1+ \varep) e_{ij}, \\
&r^{-1/2} \Vert \pr g \Vert_{\LL^{2}(B^+(x,r))} + r^{1/2} \Vert \pr^2 g \Vert_{\LL^2(B^+(x,r))} \leq \varep,
\end{aligned} \label{Linftyboundsg2222333} \end{align}
and in case of higher regularity, 
\begin{align} \label{higherregnormboundsg22223}
\Vert g \Vert_{\HH^{m+2}(B^+(x,r))} \leq C_r \sum\limits_{i=0}^m \Vert \nab^{(m)} \RRRic \Vert_{L^2(M)} + C_{r,m} \varep.
\end{align}

We define the radius $r_1>0$ by
\begin{align*}
2r_1=\min \left\{ r, \frac{\varep^4}{\Vert k \Vert_{L^4(M)}^4 + \Vert \nab k \Vert_{L^2(M)}^2} \right\},
\end{align*}
and let the chart 
\begin{align*} 
\varphi_p: B^+(0,2r_1) \to U_p \subset \Si
\end{align*}
be defined as the restriction of $\varphi$ to $B^+(0,2r_1)$.\\

The following technical lemma is used to put a coordinate system $\ol{B(0,r_1)}$ into $B^+(0,2r_1)$ such that in addition it covers an open neighbourhood of the origin in $\{ x^3=0\}$.
\begin{lemma} \label{lem:technicalpde} There is $\de_0>0$ such that for all reals $0<\de< \de_0$, there is a smooth diffeomorphism $$\Psi_\de: \ol{B(0,r_1)} \to P_\de \subset B^+(0,2r_1)$$ such that
\begin{align} \begin{aligned}
B^+(0,r_1\de) \subset \subset & P_\de,\\
\ubp(0,r_1\de) \subset \subset& P_\de,
\end{aligned}   \label{volumeboundbelwo55343} \end{align}
and for every integer $m \geq0$,
\begin{align} \label{eq3434320004}
\Vert D\Psi_\de - I \Vert_{\CC^m(\ol{B(0,r_1)})} + \Vert D(\Psi_\de)^{-1} - I \Vert_{\CC^m(\ol{P_\de})}  \leq C_{r_1,m} \de,
\end{align}
where $I$ denotes the identity matrix.
\end{lemma}
\begin{proof} The diffeomorphism $\Psi_\de$ is constructed by smoothly deforming the ball \newline $B((0,0,r_1),r_1) \subset B^+(0,2r_1)$. Details are left to the reader, see the next figure.
\begin{center}
\includegraphics[height=5cm]{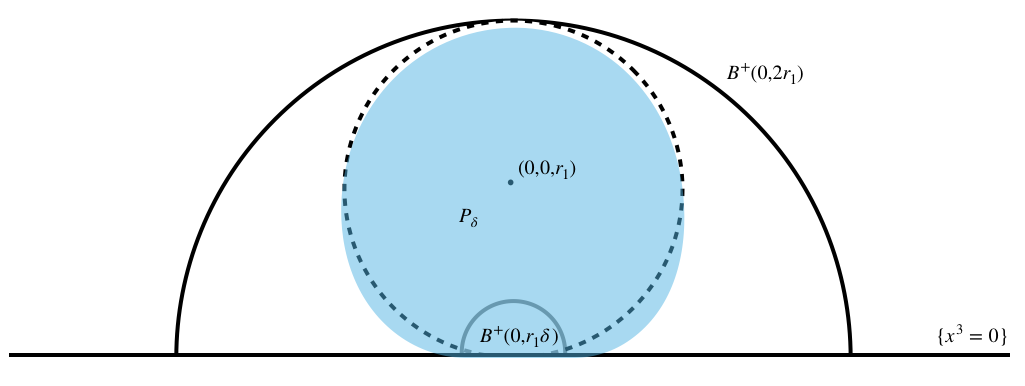} 
\captionof{figure}{The closed set $P_\de \subset B^+(0,2r_1)$ with smooth boundary is depicted as the shaded region.}
\end{center}
\end{proof}

Using $\Psi_\de$ of Lemma \ref{lem:technicalpde}, we define the chart $\varphi'_p$ as
\begin{align*}
\varphi'_p = \varphi \circ \Psi_\de: \ol{B(0,r_1)} \to V \subset \Si.
\end{align*}
Letting $\de>0$ sufficiently small depending on $\varep>0$, it holds by \eqref{Linftyboundsg2222333}, \eqref{higherregnormboundsg22223} and \eqref{eq3434320004} that for $g'= (\varphi'_p)^\ast$,
\begin{align*}
&(1-2\varep) e_{ij} \leq g'_{ij} \leq (1+ 2\varep) e_{ij} \text{ on } \ol{B(0,r_1)}, \\
&(r_1)^{1/2} \Vert \pr^2 g' \Vert_{\LL^2(B(0,r_1))} \lesssim \varep,
\end{align*}
and in case of higher regularity, for integers $m\geq1$, for $\varep>0$ sufficiently small,
\begin{align*}
\Vert g' \Vert_{\HH^{m+2}(B(0,r_1))} \leq C_{r_1,m}  \left( \sum\limits_{i=0}^m \Vert \nab^{(m)} \RRRic \Vert_{\LL^2(M)} + 1 \right).
\end{align*}

\begin{lemma} \label{lemmalowerboundvolume}
The $g$-area of $\varphi_p(P_\de) \cap \pr \Si$ is bounded from below by
\begin{align*}
\mathrm{area}_g (\varphi_p(P_\de) \cap \pr \Si) \geq (1-\varep)\left( r_1 \de \right)^2.
\end{align*}
\end{lemma}

\begin{proof} By the fact that $\ubp(0,r_1 \de) \subset P_\de$, see \eqref{volumeboundbelwo55343}, and \eqref{Linftyboundsg2222333}, we have
\begin{align*}
\mathrm{area}_g (\varphi_p(P_\de) \cap \pr \Si) \geq \mathrm{area}_g( \varphi_p(\ubp(0,r_1\de))) \geq (1-\varep) \left( r_1 \de \right)^2.
\end{align*} \end{proof}

We now turn to pick points $(p_i)_{i=1}^{N_1} \subset \Si$. Let the integer $N_1\geq1$ and $(p_i)_{i=1}^{N_1}\subset \pr \Si$ be such that 
\begin{align} \label{coverpipointschoice}
\pr \Si \subset \bigcup_{i=1}^N \varphi_{p_i}\left(\ubp\left(0,\frac{r_1 \de}{2}\right)\right).
\end{align}
By Lemma \ref{lemmalowerboundvolume}, the smallest necessary integer $N_1 \geq1$ for \eqref{coverpipointschoice} depends only on $\mathrm{area}_g(\pr \Si)$, $r_1$, $\de$ and $\varep$. Define the sets $V_i$ as
\begin{align} \label{defofVI}
V_i = \varphi'_{p_i}(\ol{B(0,r_1)}).
\end{align}

In the next lemma, we use the following definition. 
\begin{definition} For every real $s>0$, define the annulus $\AA_s$ by
\begin{align*}
\AA_{s} := \{ p \in \Si: d_g(p, \pr \Si) < s \} \subset \Si,
\end{align*}
where $d_g$ denotes the geodesic distance. \end{definition}

\begin{lemma} \label{claimcovering243434}
For $\varep>0$ sufficiently small, the constructed $(V_i)_{i=1}^{N_1}$ satisfy the following.
\begin{enumerate}
\item $\pr \Si \subset \cup_{i=1}^{N_1} V_i$.
\item For all $i\in \{1, \dots, N_1\}$, $B_g(p_i,(1-\varep) r_1 \de) \subset V_i$.
\item $\AA_{\frac{r_1 \de}{4}}  \subset \cup_{i=1}^{N_1} V_i$.
\item The Lebesgue number\footnote{Given a covering $(C_i)_{i=1}^N$ of $M$, the Lebesgue number $\ell$ is defined to be the largest real number such that for each point $p \in M$, there is an $i\in \{1, \dots, N\}$ such that $B_g(p,\ell) \subset C_i$.} $\ell$ of the cover $(V_i)_{i=1}^{N_1}$ of $\AA_{\frac{r_1 \de}{4}}$ is bounded from below by 
$$\ell \geq \frac{r_1 \de}{16}.$$
\end{enumerate}
\end{lemma}
%%%%%%%%%%%%%%%%%%%%%%%%%%%%%%%%%%%%%%%%%%%%%%%%%%%%%%%%%%

\begin{proof} {\bf Proof of (1).} By \eqref{volumeboundbelwo55343}, $\ubp\left(0,\frac{r_1\de}{2}\right) \subset P_\de$, so together with \eqref{coverpipointschoice}, we have
\begin{align*}
\pr \Si \subset  \bigcup_{i=1}^{N_1} \varphi_{p_i} \left(\ubp\left(0,\frac{r_1\de}{2}\right)\right) \subset  \bigcup_{i=1}^{N_1} \varphi_{p_i} (P_\de ) =\bigcup_{i=1}^{N_1} \varphi'_{p_i} (\ol{B(0,r_1)} ) =  \bigcup_{i=1}^{N_1} V_i.
\end{align*}

{\bf Proof of (2).} By Remark \ref{RemarkComparison}, \eqref{Linftyboundsg2222333}, \eqref{volumeboundbelwo55343} and \eqref{defofVI}, we have 
\begin{align*}
B_g(p_i,(1-\varep)r_1\de) \subset \varphi_{p_i} (B^+(0,(1-\varep^2)r_1\de)) \subset \varphi_{p_i}(B^+(0,r_1\de)) \subset V_i.
\end{align*}

{\bf Proof of (3).} Let $p \in \AA_{ \frac{r_1\de}{4}}$. By definition of $\AA_{\frac{r_1\de}{4}}$, there is a point $p' \in \pr \Si$ such that 
$$d_g(p,p') < \frac{r_1\de}{4}.$$ 
Further, by \eqref{coverpipointschoice} and Remark  \ref{RemarkComparison}, there is a $p_i \in \pr \Si$ such that 
$$d_g(p',p_i) < (1+\varep)\frac{r_1\de}{2}.$$
By the above two, using the triangle inequality,
\begin{align*}
d_g(p_i,p) \leq& d_g(p_i,p') + d_g(p',p) < (1+\varep)\frac{3r_1\de}{4}.
\end{align*}
Consequently, using (2) of this lemma, we have for $\varep>0$ sufficiently small,
\begin{align*}
p \in B_g\left(p_i,(1+ \varep)\frac{3r_1\de}{4} \right) \subset B_g\left(p_i,(1-\varep)r_1 \de \right) \subset \varphi_{p_i}(B^+(0,r_1\de)) \subset V_i \subset \cup_{i=1}^{N_1} V_i.
\end{align*}

{\bf Proof of (4).} Let $p \in \AA_{\frac{r_1\de}{4}}$, and let $\tilde{p} \in B_g\left( p, \frac{r_1\de}{4} \right)$. By definition of $\AA_{\frac{r_1\de}{4}}$, there is a point $p' \in \pr \Si$ such that 
$$d_g(p,p') < \frac{r_1\de}{4}.$$ 
Further, by \eqref{coverpipointschoice} and Remark  \ref{RemarkComparison}, there is a $p_i \in \pr \Si$ such that 
$$d_g(p',p_i) < (1+\varep)\frac{r_1\de}{2}.$$
Therefore, by using the triangle inequality,
\begin{align*}
d_g(p_i,\tilde{p}) \leq& d_g(p_i,p') + d_g(p',p) + d_g(p,\tilde{p}) \\
<& (1+\varep) \frac{r_1\de}{2} + \frac{r_1\de}{4} + \frac{r_1\de}{16} \\
<& (1+\varep) \frac{15 r_1\de}{16}.
\end{align*}
Consequently, using (2) of this lemma, for $\varep>0$ sufficiently small,
\begin{align*}
\tilde{p} \in B_g\left(p_i,(1+ \varep)\frac{15r_1\de}{16} \right) \subset B_g\left(p_i,(1-\varep)r_1 \de \right) \subset V_i.
\end{align*}

This finishes the proof of Lemma \ref{claimcovering243434}. \end{proof}

%%%%%%%%%%%%%%%%%%%%%%%%%%%%%%%%%%%%%%%%%%%%%%%%%%%%%%%%%%

{\bf Construction of coordinate balls $B(0,r_2)$ away from $\pr \Si$.} First, for a given \newline $p \in \Si \setminus \AA_{\frac{r_1\de}{4}}$, we construct a coordinate system $B(0,r_2)$ in $\Si$. Then we pick points $(p_i)_{i=N_1+1}^{N_2} \subset \Si$ such that the corresponding constructed coordinate systems cover $\Si \setminus \AA_{\frac{r_1\de}{4}}$. \\

Let thus $p\in \Si \setminus \AA_{\frac{r_1\de}{4}}$, and let $\varep>0$ small to be determined. By Theorem \ref{CGexistencetheorem}, there is a radius
$$r=r(\Vert \RRRic \Vert_{\LL^2(\Si)}, \Vert \Theta \Vert_{\LL^4(\pr \Si)}, r_{vol}(\Si,1), \varep)>0$$
and a chart $\varphi: B^+(0,r) \to U \subset \Si$ with $\varphi(0)=p$ and 
\begin{align} \begin{aligned} 
&(1-\varep) e_{ij} \leq g_{ij} \leq (1+ \varep) e_{ij}, \\
&r^{-1/2} \Vert \pr g \Vert_{\LL^{2}(B^+(x,r))} + r^{1/2} \Vert \pr^2 g \Vert_{\LL^2(B^+(x,r))} \leq \varep,
\end{aligned} \label{Linftyboundsg22223334} \end{align}
and in case of higher regularity,
\begin{align} \label{higherregnormboundsg22223444}
\Vert g \Vert_{\HH^{m+2}(B^+(x,r))} \leq C_{r} \sum\limits_{i=0}^m \Vert \nab^{(m)} \RRRic \Vert_{\LL^2(M)} + C_{r,m} \varep.
\end{align}

We define the radius $r_2>0$ by
\begin{align} \label{defofr2}
r_2=\min \left\{ r, \frac{r_1\de}{8}, \frac{\varep^4}{\Vert k \Vert_{L^4(M)}^4 + \Vert \nab k \Vert_{L^2(M)}^2} \right\}.
\end{align}
For $\varep>0$ sufficiently small, by Remark \ref{RemarkComparison} and the fact that $dist_g(p,\pr \Si) > r_2$, the chart 
\begin{align*} 
\varphi_p: B(0,r_2) \to U \subset \Si,
\end{align*}
defined as the restriction of $\varphi$ to $B(0,r_2)$, is well-defined. Moreover, by \eqref{Linftyboundsg22223334} and \eqref{higherregnormboundsg22223444},
\begin{align*} 
&(1-\varep) e_{ij} \leq g_{ij} \leq (1+ \varep) e_{ij}, \\
&(r_2)^{1/2} \Vert \pr^2 g \Vert_{\LL^2(B(0,r_2))} \leq \varep,
\end{align*}
and in case of higher regularity, for integers $m\geq1$,
\begin{align*}
\Vert g \Vert_{\HH^{m+2}(B(0,r_2))} \leq C_{r_2} \sum\limits_{i=0}^m \Vert \nab^{(m)} \RRRic \Vert_{L^2(M)} + C_{r_2,m} \varep.
\end{align*}

We now turn to pick the points $(p'_i)_{i=N_1+1}^{N_2}$. Let the integer $N_2\geq0$ and $(p_i)_{i=N_1+1}^{N_2}$ be such that 
\begin{align*}
\Si \setminus \AA_{\frac{r_1 \de}{4}}\subset \bigcup_{i=1}^N \varphi_{p_i} \left(B\left(0,\frac{r_2}{2}\right)\right).
\end{align*}
The integer $N_2\geq1$ depends only on $\mathrm{area}_g(\pr \Si), \mathrm{vol}_g \Si, \varep$ and the low regularity geometric bounds assumed in Theorem \ref{thm:bl2locintro1}. Define the sets $U_i$ as
\begin{align*}
U_i := \varphi_{p_i} (B(0,r_2)).
\end{align*}

We have the next result. Its proof is similar to Lemma \ref{claimcovering243434} and left to the reader.
\begin{lemma} For $\varep>0$ sufficiently small, the constructed $(U_i)_{i=N_1+1}^{N_2}$ satisfy the following.
\begin{itemize}
\item $\Si \setminus \AA_{\frac{r_1 \de}{4}} \subset \cup_{i=N_1+1}^{N_2} U_i$.
\item The Lebesgue number $\ell$ of the cover $(U_i)_{i=N_1+1}^{N_2}$ of $\Si \setminus \AA_{\frac{r_1 \de}{4}}$ is bounded by
\begin{align*}
\ell \geq \frac{r_2}{16}
\end{align*}
\end{itemize}
\end{lemma}

To summarise the above, we have a constructed a cover of $\Si$ by coordinate systems $\ol{B(0,r_1)}$ and $B(0,r_2)$. 
\begin{lemma} \label{FullLebesgueControl} The Lebesgue number of the constructed cover $(U_i)_{i=1}^{N_1}, (V_i)_{i=N_1+1}^{N_2}$ of $\Si$ is bounded by
\begin{align*}
\ell \geq \frac{r_2}{16}.
\end{align*}
\end{lemma}

\begin{proof} For every point $p \in \Si$, either $p \in \AA_{\frac{r_1\de}{4}}$ or $p \in \Si \setminus \AA_{\frac{r_1\de}{4}}$. Therefore, by construction of $(U_i)_{i=1}^N$, $(V_i)_{i=1}^N$ and the definition of $r_2 \leq r_1$, see \eqref{defofr2}, there exists an $i \in \{1, \dots, N_2\}$ such that either $B_g\left(p,\frac{r_2}{16}\right) \subset V_i$ or $B_g\left(p,\frac{r_2}{16}\right) \subset U_i$. \end{proof}

%%%%%%%%%%%%%%%%%%%%%%%%%%%%%%%%%%%%%%%%%%%%%%%%%%%%%%%%%%

\section{The scaling to small data} \label{sec:ReductionToSmallBalls} 

The Einstein equations \eqref{EinsteinVacuumEquationsIntroJ} are invariant under the scaling
\begin{align} \label{spacetimescaling}
\mathbf{g}(t,x) \to \mathbf{g}_\la(t,x):= \mathbf{g}(\la t, \la x),
\end{align}
where $\la>0$ is a real number. As a consequence, the constraint equations are invariant under the scaling
\begin{align} \label{initialdatarescaling}
(g,k)(x) \to (g_{(\la)}, k_{(\la)})(x):= (g,\la k)(\la x).
\end{align}

The main result of this section is the following.

\begin{lemma}[Scaling to small data] \label{lemmaScalingSmall}
Let $r>0$ and let $(B(0,r),g,k)$ be initial data such that for some $0<\varep<1$, 
\begin{align*}
&(1-\varep)e_{ij} \leq g_{ij} \leq (1+\varep)e_{ij} \\
&r^{1/2} \Vert \pr^2 g \Vert_{\LL^2(B(0,r))} \leq \varep.
\end{align*}
Then, for 
\begin{align*}
\la = \min \left\{ r, \frac{\varep^4}{\Vert k \Vert^4_{L^4(B(0,r))}+ \Vert \nab k \Vert_{L^2(B(0,r))}^2} \right\},
\end{align*}
the rescaled initial data set $(B(0,\la^{-1}r),g_{(\la)}, k_{(\la)})$ satisfies
\begin{align*}
\Vert g_{(\la)}-e \Vert_{\HH^2(B(0,\la^{-1}r))} + \Vert k \Vert_{\HH^1(B(0,\la^{-1}r))} \lesssim \varep.
\end{align*}
In case of higher regularity, we have for integers $m\geq0$
\begin{align*}
\Vert g_{(\la)}-e \Vert_{\HH^{m+2}(B(0,\la^{-1}r))} \leq C_{\la,r} \sum\limits_{i=0}^m \Vert \nab^{(m)} \RRRic \Vert_{L^2(M)} + C_{\la,r,m} \varep.
\end{align*}
\end{lemma}

\begin{proof} By the scaling \eqref{initialdatarescaling}, we have 
\begin{align}\begin{aligned}
\Vert k_{(\la)} \Vert_{\LL^2(B(0,\la^{-1}r))} \lesssim& \Vert k_{(\la)} \Vert_{\LL^4(B(0,\la^{-1}r))}\\
\lesssim& \la^{1/4} \Vert k \Vert_{\LL^4(B(0,r))}\\
\lesssim& \varep, \\
\Vert \nab k_{(\la)} \Vert_{\LL^2(B(0,\la^{-1}r))} =& \la^{1/2} \Vert \nab k \Vert_{\LL^2(B(0,r))} \\
\leq& \varep, \\
\Vert \pr^2 g_{(\la)} \Vert_{\LL^2(B(0,\la^{-1}r))} =& \la^{1/2} \Vert \pr^2 g \Vert_{\LL^2(B(0,r))} \\
\leq& \la^{1/2} \frac{\varep}{r^{1/2}}\\
\leq& \varep.
\end{aligned} \label{scaling1222} \end{align}

Moreover, we have on $B(0,\la^{-1}r)$
\begin{align} \label{scaling1222222}
(1-\varep)e_{ij} \leq (g_{(\la)})_{ij} \leq (1+\varep) (g_{(\la)})_{ij}.
\end{align}
It follows from \eqref{scaling1222} and \eqref{scaling1222222} by interpolation that
\begin{align*}
\Vert g_{(\la)}-e \Vert_{\LL^2(B(0,\la^{-1}r))} + \Vert \pr g_{(\la)} \Vert_{\LL^2(B(0,\la^{-1}r))} \lesssim \varep.
\end{align*}
The higher regularity estimate is similar and left to the reader. This finishes the proof of Lemma \ref{lemmaScalingSmall}.
\end{proof}

%%%%%%%%%%%%%%%%%%%%%%%%%%%%%%%%%%%%%%%%%%%%%%%%%%%%%%%%%%

\section{The conclusion of the proof of Theorem \ref{thm:bl2locintro1}} \label{SectionConclusionOfTheorem} 

In this section we conclude the proof of Theorem \ref{thm:bl2locintro1} by combining the results of the previous sections. We first have the following proposition.
\begin{proposition} \label{ingredient1}
Let $(\Si,g,k)$ be a maximal initial data set such that $(\Si,g)$ is a compact Riemannian manifold with boundary and
\begin{align*}
\Vert \RRRic \Vert_{L^2(\Si)} < \infty, \Vert \Th \Vert_{L^4(\pr \Si)} <\infty, \Vert k \Vert_{L^4(\Si)} < \infty, \Vert \nab k \Vert_{L^2(\Si)} < \infty, r_{vol}(\Si,1)>0.
\end{align*}
Let
\begin{align*}
( \varphi_i: \ol{B(0,r_1)} \to V_i \subset \Si )_{i=1}^{N_1},
( \varphi_i: B(0,r_2) \to U_i \subset \Si )_{i=N_1+1}^{N_2}
\end{align*}
be the constructed cover of $\Si$, see Section \ref{subsec5893434}. Then, the future domains of dependence $\DD(V_i)$ and $\DD(U_i)$ are each foliated by spacelike maximal hypersurfaces $\Si_t$ defined as level sets of a time function $t$ with $\Si_0=U_i$ (or $\Si_0=V_i$) such that
\begin{enumerate}
\item {\bf $L^2$-regularity.} there is a constant 
\begin{align*}
C=C(\Vert \RRRic \Vert_{L^2(\Si)}, \Vert k \Vert_{L^4(\Si)}, \Vert \nab k \Vert_{L^2(\Si)}, \Vert \Th \Vert_{L^4(\pr \Si)}, r_{vol}(\Si,1))>0
\end{align*}
and a time 
\begin{align*}
T=T(\Vert \RRRic \Vert_{L^2(\Si)}, \Vert k \Vert_{L^4(\Si)}, \Vert \nab k \Vert_{L^2(\Si)}, \Vert \Th \Vert_{L^4(\pr \Si)}, r_{vol}(\Si,1))>0
\end{align*}
such that on $0 \leq t \leq T$, we have
\begin{align} \label{quantbound12222}
\Vert \RRRic_t \Vert_{L^\infty_tL^2(\Si_t)} \leq C, \Vert k_t \Vert_{L^\infty_tL^4(\Si_t)} \leq C, \Vert \nab k_t \Vert_{L^\infty_t L^2(\Si_t)} \leq C, \inf\limits_{0\leq t \leq T}r_{vol}(\Si_t,1) \geq \frac{1}{C}.
\end{align}
\item {\bf Higher regularity.} In case of higher regularity, for integers $m\geq0$, on $0 \leq t \leq T$,
\begin{align} \label{quantbound22222}
\sum\limits_{\vert \a \vert \leq m} \Vert \D^{(\a)} \mathbf{R} \Vert_{L^\infty_t L^2(\Si_t)} \leq C_m \left( \sum\limits_{\vert i \vert \leq m} \Vert \nab^{(i)} \RRRic \Vert_{L^2(\Si)} + \Vert \nab^{(i)} \nab k \Vert_{L^2(\Si)}+1 \right).
\end{align}
\end{enumerate}
\end{proposition}

The above proposition implies the proof of Theorem \ref{thm:bl2locintro1} as follows.
\begin{proof}[Proof of Theorem \ref{thm:bl2locintro1}] Let $(\Si,g,k)$ be maximal initial data such that $(\Si,g)$ is a compact Riemannian manifold with boundary such that 
\begin{align*}
\Vert \RRRic \Vert_{L^2(\Si)} < \infty, \Vert \Th \Vert_{L^4(\pr \Si)} <\infty, \Vert k \Vert_{L^4(\Si)} < \infty, \Vert \nab k \Vert_{L^2(\Si)} < \infty, r_{vol}(\Si,1)>0.
\end{align*}
Let $(V_i)_{i=1}^{N_1}$ and $(U_i)_{i=N_1+1}^{N_2}$ be the cover of $\Si$ constructed in Section \ref{subsec5893434}. \\

On the one hand, by Proposition \ref{ingredient1}, it follows that the future domains of dependence of $V_i$ and $U_i$ are controlled with quantitative bounds \eqref{quantbound12222} and \eqref{quantbound22222}. \\

On the other hand, by the construction of the cover, see Lemma \ref{FullLebesgueControl}, there is 
$$\tilde{r}=\tilde{r}(\Vert \RRRic \Vert_{L^2(\Si)}, \Vert k \Vert_{L^4(\Si)}, \Vert \nab k \Vert_{L^2(\Si)}, \Vert \Th \Vert_{L^4(\pr \Si)}, r_{vol}(\Si,1))>0$$
such that for every $p\in \Si$, there is a $V_i$ or $U_i$ such that
$$B_g(p,\tilde{r}) \subset V_i \text{ or } B_g(p,\tilde{r}) \subset U_i.$$
In particular, it holds that the future domain of dependence 
$$\DD(B_g(p,r)) \subset \DD(U_i) \text{ or }\DD(B_g(p,r)) \subset \DD(V_i),$$ 
and thus is foliated by spacelike maximal hypersurfaces $\Si_t$ defined as level sets of a time function $t$ with bounds \eqref{quantbound12222} and \eqref{quantbound22222}. We remark that the control of the volume radius follows as in the proof of Theorem \ref{thmL2smalldata} by a control of $\bf{g}_{\mu \nu}$ in $C^0$. We refer the reader to the estimates in Section 4 of \cite{J3}. This finishes the proof of Theorem \ref{thm:bl2locintro1}.
\end{proof}

It remains to prove Proposition \ref{ingredient1}.
\begin{proof}[Proof of Proposition \ref{ingredient1}] Consider the future domain of dependence $\DD(U_i)$ of $U_i \subset \Si$. By using the scaling \eqref{initialdatarescaling} with $\la = r_2$, it follows by Lemma \ref{lemmaScalingSmall} and our choice of $r_2>0$ in \eqref{defofr2}, that the rescaled initial data $(B(0,1),g_{(r_2)},k_{(r_2)})$ satisfies
\begin{align*}
\Vert g_{(r_2)}-e \Vert_{\HH^2(B(0,1))} + \Vert k_{(r_2)} \Vert_{\HH^1(B(0,r))} \lesssim \varep,
\end{align*}
and in case of higher regularity, for integers $m\geq1$,
\begin{align*}
\Vert g_{(r_2)}-e \Vert_{\HH^{m+2}(B(0,1))} + \Vert k_{(r_2)} \Vert_{\HH^{m+1}(B(0,1))} \leq& C_{r_2,m} \Big( \sum\limits_{i=1}^m \Vert \nab^{(i)}\RRRic \Vert_{L^2(\Si)} + \Vert \nab^{(i)} \nab k \Vert_{L^2(\Si)} +1 \Big).
\end{align*}
Therefore for $\varep>0$ sufficiently small, by Proposition \ref{prop:red}, the future domain of dependence of $B(0,1)$ is locally foliated by spacelike maximal hypersurfaces $\Si_t$ defined as level sets of a time function $t$ such that on $0 \leq t \leq 1/2$,
\begin{align*} 
\Vert \RRRic_{(r_2)} \Vert_{L^\infty_{t}L^2(\Si_t)} \lesssim \varep,  \Vert k_{(r_2)} \Vert_{L^\infty_{t}L^4(\Si_t)}\lesssim \varep, \Vert \nab k_{(r_2)} \Vert_{L^\infty_{t}L^2(\Si_t)} \lesssim \varep, \inf\limits_{0\leq t \leq 1/2}  r_{vol}(\Si_t,1) \geq 1/8,
\end{align*}
and in case of higher regularity, for integers $m\geq1$, on $0 \leq t \leq 1/2$,
\begin{align*}
\sum\limits_{\vert \a \vert \leq m} \Vert \D^{(\a)} \mathbf{R}_{(r_2)} \Vert_{L^\infty_{t}L^2(\Si_t)}
\leq& C_m \Big(  \Vert g_{(r_2)}-e \Vert_{\HH^{m+2}(B(0,1))} + \Vert k_{(r_2)} \Vert_{\HH^{m+1}(B(0,1))} \Big) \\
\leq& C_{r_2,m} \Big(\sum\limits_{i=1}^m \Vert \nab^{(i)}\RRRic \Vert_{L^2(\Si)} + \Vert \nab^{(i)} \nab k \Vert_{L^2(\Si)} +1 \Big).
\end{align*}

Using the spacetime scaling \eqref{spacetimescaling} with $\la = (r_2)^{-1}$, it follows that the future domain of dependence of $B(0,r_2)$ is controlled up to time 
$$T=\frac{r_2}{2}=\frac{r_2}{2}(\Vert \RRRic \Vert_{L^2(\Si)}, \Vert k \Vert_{L^4(\Si)}, \Vert \nab k \Vert_{L^2(\Si)}, \Vert \Th \Vert_{L^4(\pr \Si)}, r_{vol}(\Si,1))>0$$
with bounds on the interval $0 \leq t \leq T$,
\begin{align*} 
\Vert \RRRic_t \Vert_{L^\infty_t L^2(\Si_t)} \leq C, \Vert k_t \Vert_{L^\infty_tL^4(\Si_t)} \leq C, \Vert \nab k_t \Vert_{L^\infty_tL^2(\Si_t)} \leq C, \inf\limits_{0\leq t \leq T}r_{vol}(\Si_t,1) \geq \frac{1}{C},
\end{align*}
where 
\begin{align} \label{constantref}
C=C(\Vert \RRRic \Vert_{L^2(\Si)}, \Vert k \Vert_{L^4(\Si)}, \Vert \nab k \Vert_{L^2(\Si)}, \Vert \Th \Vert_{L^4(\pr \Si)}, r_{vol}(\Si,1))>0.
\end{align}
Moreover, in case of higher regularity, for integers $m\geq0$, on $0 \leq t \leq T$,
\begin{align*}
\sum\limits_{\vert \a \vert \leq m} \Vert \D^{(\a)} \mathbf{R} \Vert_{L^\infty_t L^2(\Si_t)} \leq C_m \left( \sum\limits_{\vert i \vert \leq m} \Vert \nab^{(i)} \RRRic \Vert_{L^2(\Si)} + \Vert \nab^{(i)} \nab k \Vert_{L^2(\Si)}+1 \right),
\end{align*}
where the constant $C_m>0$ depends on $m$ and the previous $C>0$ in \eqref{constantref}.\\

This finishes the control of the future domain of dependence $\DD(U_i)$. The control of $\DD(V_i)$ is similar and left to the reader. This finishes the proof of Proposition \ref{ingredient1}. \end{proof}

%%%%%%%%%%%%%%%%%%%%%%%%%%%%%%%%%%%%%%%%%%%%%%%%%%%%%%%%%%
%%%%%%%%%%%%%%%%%%%%%%%%%%%%%%%%%%%%%%%%%%%%%%%%%%%%%%%%%%
%%%%%%%%%%%%%%%%%%%%%%%%%%%%%%%%%%%%%%%%%%%%%%%%%%%%%%%%%%
%%%%%%%%%%%%%%%%%%%%%%%%%%%%%%%%%%%%%%%%%%%%%%%%%%%%%%%%%%
%%%%%%%%%%%%%%%%%%%%%%%%%%%%%%%%%%%%%%%%%%%%%%%%%%%%%%%%%%
%%%%%%%%%%%%%%%%%%%%%%%%%%%%%%%%%%%%%%%%%%%%%%%%%%%%%%%%%%
%%%%%%%%%%%%%%%%%%%%%%%%%%%%%%%%%%%%%%%%%%%%%%%%%%%%%%%%%%

\end{document}